\documentclass[runningheads,a4paper]{llncs}

\usepackage{amssymb}
\usepackage{graphicx}

\usepackage{url}
\urldef{\mailsa}\path|E-mail:devsi.bantva@gmail.com|
\newcommand{\keywords}[1]{\par\addvspace\baselineskip
\noindent\keywordname\enspace\ignorespaces#1}

\newcommand{\be}{\begin{equation}}
\newcommand{\ee}{\end{equation}}
\newcommand{\bea}{\begin{eqnarray}}
\newcommand{\eea}{\end{eqnarray}}
\newcommand{\bean}{\begin{eqnarray*}}
\newcommand{\eean}{\end{eqnarray*}}
\def\non{\nonumber}

\begin{document}
\mainmatter  

\title{On hamiltonian colorings of block graphs}

\author{Devsi Bantva}
\institute{Lukhdhirji Engineering College, Morvi - 363 642 \\
Gujarat (INDIA) \\
\mailsa\\}

\toctitle{On hamiltonian colorings of block graphs}
\tocauthor{Devsi Bantva}
\maketitle

\begin{abstract}
A hamiltonian coloring $c$ of a graph $G$ of order $p$ is an assignment of colors to the vertices of $G$ such that $D(u, v)$ + $|c(u) - c(v)|$ $\geq$ $p - 1$ for every two distinct vertices $u$ and $v$ of $G$, where $D(u, v)$ denotes the detour distance between $u$ and $v$. The value \emph{hc(c)} of a hamiltonian coloring $c$ is the maximum color assigned to a vertex of $G$. The hamiltonian chromatic number, denoted by $hc(G)$, is the min\{$hc(c)$\} taken over all hamiltonian coloring $c$ of $G$. In this paper, we present a lower bound for the hamiltonian chromatic number of block graphs and give a sufficient condition to achieve the lower bound. We characterize symmetric block graphs achieving this lower bound. We present two algorithms for optimal hamiltonian coloring of symmetric block graphs.
\end{abstract}

\section{Introduction}

A \emph{hamiltonian coloring} $c$ of a graph $G$ of order $p$ is an assignment of colors (non-negative integers) to the vertices of $G$ such that $D(u, v)$ + $|c(u) - c(v)|$ $\geq$ $p - 1$ for every two distinct vertices $u$ and $v$ of $G$, where $D(u, v)$ denotes the detour distance which is the length of the longest path between $u$ and $v$. The value of $hc(c)$ of a hamiltonian coloring $c$ is the maximum color assigned to a vertex of $G$. The \emph{hamiltonian chromatic number} $hc(G)$ of $G$ is min\{$hc(c)$\} taken over all hamiltonian coloring $c$ of $G$. It is clear from definition that two vertices $u$ and $v$ can be assigned the same color only if $G$ contains a hamiltonian $u-v$ path. Moreover, if $G$ is a hamiltonian-connected graph then all the vertices can be assigned the same color. Thus the hamiltonian chromatic number of a connected graph $G$ measures how close $G$ is to being hamiltonian-connected, minimum the hamiltonian chromatic number of a connected graph $G$ is, the closer $G$ is to being hamiltonian-connected. The concept of hamiltonian coloring was introduced by Chartrand \emph{et al.} \cite{Chartrand1} as a variation of \emph{radio k-coloring} of graphs.

At present, the hamiltonian chromatic number is known only for handful of graph families. Chartrand \emph{et al.} investigated the exact hamiltonian chromatic numbers for complete graph $K_{n}$, cycle $C_{n}$, star $K_{1,n}$ and complete bipartite graph $K_{r,s}$ in \cite{Chartrand1,Chartrand2}. Also an upper bound for $hc(P_{n})$ was established by Chartrand \emph{et al.} in \cite{Chartrand1} but the exact value of $hc(P_{n})$ which is equal to the radio antipodal number $ac(P_{n})$ given by Khennoufa and Togni in \cite{Khennoufa}. In \cite{Shen}, Shen \emph{et al.} have discussed the hamiltonian chromatic number for graphs $G$ with max\{$D(u,v)$ : $u,v \in V(G)$, $u \neq v$\} $\leq$ $\frac{p}{2}$, where $p$ is the order of graph $G$ and they gave the hamiltonian chromatic number for a special class of caterpillars and double stars. The researchers emphasize that determining the hamiltonian chromatic number is interesting but a challenging task even for some basic graph families.

Without loss of generality, we initiate with label 0, then the \emph{span} of any hamiltonian coloring $c$ which is defined as max\{$|c(u) - c(v)|$ : $u, v \in V(G)$\}, is the maximum integer used for coloring. However, in \cite{Chartrand1,Chartrand2,Shen} only positive integers are used as colors. Therefore, the hamiltonian chromatic number defined in this article is one less than that defined in \cite{Chartrand1,Chartrand2,Shen} and hence we will make necessary adjustment when we present the results of \cite{Chartrand1,Chartrand2,Shen} in this article. Moreover, for standard graph theoretic terminology and notation we follow \cite{West}.

In this paper, we present a lower bound for the hamiltonian chromatic number of block graphs and give a sufficient condition to achieve the lower bound. As an illustration, we present symmetric block graphs (those block graphs whose all blocks are cliques of size $n$, each cut vertex is exactly in $k$ blocks and the eccentricity of end vertices is same) achieving this lower bound. We present two algorithms for optimal hamiltonian coloring of symmetric block graphs.

\section{A lower bound for hamiltonian chromatic number of block graphs}

\indent\indent A \emph{block graph} is a connected graph all of whose blocks are cliques. The detour distance between $u$ and $v$, denoted by $D(u,v)$, is the longest distance between $u$ and $v$ in $G$. The \emph{detour eccentricity} $\epsilon_{D}(v)$ of a vertex $v$ is the detour distance from $v$ to a vertex farthest from $v$. The \emph{detour center} $C_{D}(G)$ of $G$ is the subgraph of $G$ induced by the vertex/vertices of $G$ whose detour eccentricity is minimum. In \cite{Chartrand4}, Chartrand \emph{et al.} shown that the detour center $C_{D}(G)$ of every connected graph $G$ lies in a single block of $G$. The vertex/vertices of detour center $C_{D}(G)$ are called \emph{detour central vertex/vertices} for graph $G$. In a block graph $G$, if $u$ is on the $w-v$ path, where $w$ is the nearest detour central vertex for $v$, then $u$ is an \emph{ancestor} of $v$, and $v$ is a \emph{descendent} of $u$. Let $u_{i}$, $i =1,2,...,n$ are adjacent vertices of a block attached to a central vertex. Then the subgraph induced by $u_{i}$, $i=1,2,...,n$ and all its descendent is called a \emph{branch} at $w$. Two branches are called \emph{different} if they are induced by vertices of two different blocks attached to the same central vertex, and \emph{opposite} if they are induced by vertices of two different blocks attached to different central vertices. For a block graph $G$, define \emph{detour level function} $\mathcal{L}$ on $V(G)$ by

\begin{center}
$\mathcal{L}(u)$ := min\{$D(w, u)$ : $w \in V(C_{D}(G))$\}, for any $u$ $\in$ $V(G)$.
\end{center}

The \emph{total detour level of a graph} $G$, denoted by $\mathcal{L}(G)$, is defined as

\begin{equation}\label{total:l}
\mathcal{L}(G) := \displaystyle \sum_{u \in V(G)} \mathcal{L}(u).
\end{equation}

Note that if $|C_{D}(G)|$ = $\omega$ then the detour distance between any two vertices $u$ and $v$ in a block graph $G$ satisfies
\begin{equation}\label{eqn:duv}
D(u,v) \leq \mathcal{L}(u) + \mathcal{L}(v) + \omega - 1.
\end{equation}

Moreover, equality holds in (\ref{eqn:duv}) if $u$ and $v$ are in different branches when $\omega$ = 1 and in opposite branches when $\omega \geq 2$.

Define $\xi$ = min\{$|V(B_{i})|-1$:$B_{i}$ is a block attached to detour central vertex\} when $\omega$ = 1; otherwise $\xi$ = 0.

We first give a lower bound for the hamiltonian chromatic number of block graphs. A hamiltonian coloring $c$ on $V(G)$, induces an ordering of $V(G)$, which is a line up of the vertices with equal or increasing images. We denote this ordering by $V(G)$ = \{$u_{0}$, $u_{1}$, $u_{2}$, ..., $u_{p-1}$\} with

0 = $c(u_{0})$ $\leq$ $c(u_{1})$ $\leq$ $c(u_{2})$ $\leq$ ... $\leq$ $c(u_{p-1})$.

Notice that, $c$ is a hamiltonian coloring, then the span of $c$ is $c(u_{p-1})$.

\begin{theorem}\label{hc:lower} Let $G$ be a block graph of order $p$ and $\omega$, $\xi$ and $\mathcal{L}(G)$ are defined as earlier then
\begin{equation}\label{eqn:lower}
hc(G) \geq (p - 1)(p-\omega) - 2 \mathcal{L}(G) + \xi.
\end{equation}
\end{theorem}
\begin{proof} It suffices to prove that any hamiltonian coloring $c$ of block graph $G$ has no span less than the right hand side of (\ref{eqn:lower}). Let $c$ be an arbitrary hamiltonian coloring for $G$, where 0 = $c(u_{0})$ $\leq$ $c(u_{1})$ $\leq$ $c(u_{2})$ $\leq$ ... $\leq$ $c(u_{p-1})$. Then $c(u_{i+1}) - c(u_{i}) \geq p - 1 - D(u_{i}, u_{i+1})$, for all 0 $\leq$ $i$ $\leq$ $p-2$. Summing up these $p-1$ inequalities, we get
\begin{equation}\label{eqn:sum}
c(u_{p-1}) - c(u_{0}) \geq (p - 1)^{2} - \displaystyle\sum_{i = 0}^{p-1} D(u_{i}, u_{i+1})
\end{equation}
We consider following two cases. \\
\textbf{Case-1:} $\omega$ = 1.~~In this case, note that $\mathcal{L}(u_{0})+\mathcal{L}(u_{p-1}) \geq \xi$ and by substituting (\ref{eqn:duv}) into (\ref{eqn:sum}) we obtain,
$$
\begin{array}{ll}
c(u_{p-1}) - c(u_{0}) & \geq (p - 1)^{2} - \displaystyle\sum_{i = 0}^{p-1} D(u_{i}, u_{i+1}) \\ [0.3cm]
 & \geq (p - 1)^{2} - \displaystyle\sum_{i = 0}^{p-1} \left(\mathcal{L}(u_{i}) + \mathcal{L}(u_{i+1}) + \omega - 1 \right) \\ [0.3cm]
 & = (p - 1)^{2} - 2\displaystyle\sum_{u \in V(G)} \mathcal{L}(u) + \mathcal{L}(u_{0}) + \mathcal{L}(u_{p-1}) - (p-1)(\omega-1)  \\ [0.3cm]
 & \geq (p - 1)(p-\omega) - 2 \mathcal{L}(G) + \xi \\ [0.3cm]
\end{array}
$$
\textbf{Case-2:} $\omega \geq 2$.~~In this case, note that $\mathcal{L}(u_{0})+\mathcal{L}(u_{p-1}) \geq 0$ and by substituting (\ref{eqn:duv}) into (\ref{eqn:sum}) we obtain,
$$
\begin{array}{ll}
c(u_{p-1}) - c(u_{0}) & \geq (p - 1)^{2} - \displaystyle\sum_{i = 0}^{p-1} D(u_{i}, u_{i+1}) \\ [0.3cm]
 & \geq (p - 1)^{2} - \displaystyle\sum_{i = 0}^{p-1} \left(\mathcal{L}(u_{i}) + \mathcal{L}(u_{i+1}) + \omega - 1 \right) \\ [0.3cm]
 & = (p - 1)^{2} - 2\displaystyle\sum_{u \in V(G)} \mathcal{L}(u) + \mathcal{L}(u_{0}) + \mathcal{L}(u_{p-1}) - (p-1)(\omega-1)  \\ [0.3cm]
 & \geq (p - 1)(p-\omega) - 2 \mathcal{L}(G) \\ [0.3cm]
 & = (p - 1)(p-\omega) - 2 \mathcal{L}(G) + \xi
\end{array}
$$
Thus, any hamiltonian coloring has span not less than the right hand side of (\ref{eqn:lower}) and hence we obtain $hc(G) \geq (p - 1)(p-\omega) - 2 \mathcal{L}(G) + \xi$.
\end{proof}

The next result gives sufficient condition with optimal hamiltonian coloring for the equality in (\ref{eqn:lower}).

\begin{theorem}\label{hc:main} Let $G$ be a block graph of order $p$ and, $\omega$, $\xi$ and $\mathcal{L}(G)$ are defined as earlier then
\begin{equation}\label{eqn:main}
hc(G) = (p - 1)(p-\omega) - 2 \mathcal{L}(G) + \xi,
\end{equation}
if there exists an ordering \{$u_{0}$, $u_{1}$,...,$u_{p-1}$\} with 0 = $c(u_{0}) \leq c(u_{1}) \leq ... \leq c(u_{p-1})$ of vertices of block graph $G$ such that
\begin{enumerate}
\item $\mathcal{L}(u_{0})$ = 0, $\mathcal{L}(u_{p-1})$ = $\xi$ when $\omega$ = 1 and $\mathcal{L}(u_{0})$ = $\mathcal{L}(u_{p-1})$ = 0 when $\omega \geq 2$,
\item $u_{i}$ and $u_{i+1}$ are in different branches when $\omega = 1$ and opposite branches when $\omega \geq 2$,
\item $D(u_{i}, u_{i+1}) \leq \frac{p}{2}$, for $0 \leq i \leq p-2$.
\end{enumerate}
Moreover, under these conditions the mapping $c$ defined by
\begin{equation}\label{f0}
c(u_{0}) = 0
\end{equation}
\begin{equation}\label{f1}
c(u_{i+1}) = c(u_{i}) + p - 1 - \mathcal{L}(u_{i}) - \mathcal{L}(u_{i+1}) - \omega + 1, 0 \leq i \leq p-2
\end{equation}
is an optimal hamiltonian coloring of $G$.
\end{theorem}
\begin{proof} Suppose (1), (2) and (3) hold for an ordering \{$u_{0}$, $u_{1}$,...,$u_{p-1}$\} of the vertices of $G$ and $c$ is defined by (\ref{f0}) and (\ref{f1}). By Theorem \ref{hc:lower}, it is enough to prove that $c$ is a hamiltonian coloring whose span is $c(u_{p-1})$ = $(p - 1)(p-\omega) - 2 \mathcal{L}(G)$ + $\xi$.

Without loss of generality assume that $j-i \geq 2$ then
$$
\begin{array}{ll}
c(u_{j}) - c(u_{i}) & = \displaystyle\sum_{t = i}^{j-1} [c(u_{t+1}) - c(u_{t})] \\ [0.3cm]
 & \geq \displaystyle\sum_{t = i}^{j-1} [p - 1 - \mathcal{L}(u_{t}) - \mathcal{L}(u_{t+1}) - w + 1] \\ [0.3cm]
 & \geq \displaystyle\sum_{t = i}^{j-1} [p - 1 - D(u_{t}, u_{t+1})] \\ [0.3cm]
 & = (j-i)(p-1) - \displaystyle\sum_{t = i}^{j-1} D(u_{t}, u_{t+1}) \\ [0.3cm]
 & \geq (j-i)(p-1) - (j-i)(\frac{p}{2}) \\ [0.3cm]
 & = (j-i)\left(\frac{p-1}{2}\right) \\ [0.3cm]
 & = p-2
\end{array}
$$
Note that $D(u_{i},u_{i+1}) \geq 1$; it follows that $|c(u_{j})-c(u_{i})| + D(u_{i},u_{i+1}) \geq p-1$. Hence, $c$ is a hamiltonian coloring for $G$. The span of $c$ is given by
$$
\begin{array}{ll}
\mbox{span}(c) & = \displaystyle\sum_{t = 0}^{p-2} [c(u_{t+1}) - c(u_{t})] \\ [0.3cm]
 & = \displaystyle\sum_{t = 0}^{p-2} [p - 1 - \mathcal{L}(u_{t}) - \mathcal{L}(u_{t+1}) - \omega + 1] \\ [0.3cm]
 & = (p-1)^{2} - \displaystyle\sum_{t = 0}^{p-2} [\mathcal{L}(u_{t})+\mathcal{L}(u_{t+1})] - (p-1)(\omega-1) \\ [0.3cm]
 & = (p-1)(p-\omega) - 2 \displaystyle\sum_{u \in V(G)} \mathcal{L}(u) + \mathcal{L}(u_{0}) + \mathcal{L}(u_{p-1}) \\ [0.3cm]
 & = (p-1)(p-\omega) - 2 \mathcal{L}(G) + \xi \\ [0.3cm]
\end{array}
$$
Therefore, $hc(G) \leq (p - 1)(p-\omega) - 2 \mathcal{L}(G) + \xi$. This together with (\ref{eqn:lower}) implies (\ref{eqn:main}) and that $c$ is an optimal hamiltonian coloring.
\end{proof}

\section{Hamiltonian chromatic number of symmetric block graphs}

\indent\indent In this section, we continue to use the terminology and notation defined in previous section. We use Theorem \ref{hc:lower} and \ref{hc:main} to determine the hamiltonian chromatic number of symmetric block graphs.

A \emph{symmetric block graph}, denoted by $B_{n,k}$(or $B_{n,k}(d)$ if diameter is $d$), is a block graph with at least two blocks such that all blocks are cliques of size $n$, each cut vertex is exactly in $k$ blocks and the eccentricity of end vertices is same (see Figure \ref{Fig1}). It is straight forward to verify that the detour center of symmetric block graph of diameter $d$ is a vertex when $d$ is even and a block of size $n$ when $d$ is odd. Consequently, the number of detour central vertex/vertices for a symmetric block graph $B_{n,k}$ of diameter $d$ is either 1 or $n$ depending upon $d$ is even or odd. We observe that $B_{2,k}(2)$ are stars $K_{1,k}$, $B_{n,k}(2)$ are one point union of $k$ complete graphs (a one point union of $k$ complete graphs, also denoted by $K_{n}^{k}$, is a graph obtained by taking $v$ as a common vertex such that any two copies of $K_{n}$ are edge disjoint and do not have any vertex common except $v$), $B_{2,2}(d)$ are paths $P_{d+1}$ and $B_{2,k}(d)$ are symmetric trees (see \cite{Vaidya}). The hamiltonian chromatic number of stars $K_{1,k}$ is reported by Chartrand \emph{et al.} in \cite{Chartrand1}. The hamiltonian chromatic number of paths which is equal to the antipodal radio number of paths given by Khennoufa and Togni in \cite{Khennoufa} and the hamiltonian chromatic number of symmetric trees is investigated by Bantva in \cite{Bantva}. Hence we consider $k \geq 2$ and $d,n \geq 3$. However, for completeness we first give the hamiltonian chromatic number for $B_{n,k}(2)$ in Theorem \ref{one:thm} and next we consider general case.

\begin{theorem}\cite{Chartrand1}\label{star:thm} For $n \geq 3$, $hc(K_{1,n})$ = $(n-1)^{2}$.
\end{theorem}

\begin{theorem}\cite{Khennoufa}\label{path:thm} For any $n \geq 5$,
\begin{eqnarray*}
hc(P_{n}) = ac(P_{n}) & = &\left\{
\begin{array}{ll}
2p^{2}-2p+2, & $if $ n = 2p+1, \\ [0.3cm]
2p^{2}-4p+4, & $if $ n = 2p.
\end{array}
\right.
\end{eqnarray*}
\end{theorem}

\begin{theorem}\cite{Bantva}\label{tree:thm} Let $T$ be a symmetric tree of order $p \geq 4$ and $\Delta(T) \geq 3$. Then
$$
hc(T) = (p-1)(p-1-\epsilon(T))+\epsilon^{'}(T)-2\mathcal{L}(T),
$$
where $\epsilon(T)$ = 0 when $C(T)$ = \{$w$\} and $\epsilon(T)$ = 1 when $C(T)$ = \{$w,w^{'}$\}; and $\epsilon^{'}(T)$ = $1-\epsilon(T)$
\end{theorem}

The next result gives the hamiltonian chromatic number for one point union of $k$ copies of complete graph $K_{n}$.

\begin{theorem}\label{one:thm} For $n, k \geq 2$,
\begin{eqnarray*}
hc(K_{n}^{k}) & = &\left\{
\begin{array}{ll}
(n-1)^{2}, & $if $ k = 2, \\ [0.3cm]
k(k-2)(n-1)^{2}+n-1, & $if $ k \geq 3.
\end{array}
\right.
\end{eqnarray*}
\end{theorem}
\begin{proof} Let $K_{n}^{k}$ be one point union of $k$ complete graph. To prove the result we consider following two cases.

\textbf{Case - 1:} $k$ = 2.~~Let $G$ = $K_{n}^{2}$ with vertex set \{$x_{1}$, $x_{2}$, ..., $x_{n-1}$, $y_{1}$, $y_{2}$, ..., $y_{n-1}$, $z$\}, where $x_{i}$ and $y_{i}$, $1 \leq i \leq n-1$ be the vertices of block on each side and $z$ is the common vertex of two blocks in $G$. Let $c$ be a minimum hamiltonian coloring of $G$ with 0 $\in c(V(G))$. Since $G$ contains hamiltonian path between $x_{i}$ and $y_{i}$ for $1 \leq i \leq n-1$, we can color $x_{i}$ and $y_{i}$ with same color. Since $D(z, x_{i})$ = $D(z, y_{i})$ = $D(x_{i}, x_{j})$ = $D(y_{i}, y_{j})$ = $n-1$ and $D(x_{i}, y_{j})$ = $2n-2$ = $p-1$, for $1 \leq i, j \leq n-1$ and $i \neq j$. It follows that $|c(z) - c(x_{j})| \geq n-1$ and $|c(x_{i}) - c(x_{j})| \geq n-1$. This implies that $hc(G)$ = $hc(c)$ $\geq$ $0 + (n-1)(n-1)$ = $(n-1)^{2}$.

Next we show that $hc(G)$ $\leq$ $(n-1)^{2}$. To prove this, it is enough to give hamiltonian coloring with span equal to $(n-1)^{2}$. Define a coloring $c$ of $G$ by

$c(z)$ = 0 \\
\indent$c(x_{i})$ = $c(y_{i})$ = $i(n-1)$, $1 \leq i \leq n-1$

Since $c$ is a hamiltonian coloring, $hc(G)$ $\leq$ $hc(c)$ = $c(x_{n-1})$ = $c(y_{n-1})$ = $(n-1)(n-1)$ = $(n-1)^{2}$ and hence $hc(G)$ = $hc(K_{n}^{(2)})$ = $(n-1)^{2}$.

\textbf{Case - 2:} $k \geq 3$.~~Let $G$ = $K_{n}^{k}$ with vertex set $\{v_{i}^{j}, w : 1 \leq i \leq n-1, 1 \leq j \leq k\}$ such that for each $j$ = $1, 2, ..., k$, $v_{i}^{j}$ where $1 \leq i \leq n-1$ are in same block and $w$ is the common vertex of $G$. Define a coloring $c$ of $G$ by

$c(w) = 0$ \\
\indent$c(v_{1}^{1})$ = $(k-1)(n-1)$ \\
\indent$c(v_{i+1}^{1})$ = $c(v_{i}^{1}) + k(n-1)$, $2 \leq i \leq n-1$ \\
\indent For $j$ = 1, 2, ..., $k-1$ \\
\indent$c(v_{i}^{j+1})$ = $c(v_{i}^{^{j}}) + (k-2)(n-1)$, $1 \leq i \leq n-1$.

Since $c$ is a hamiltonian coloring, $hc(G)$ $\leq$ $hc(c)$ = $(k-1)(n-1) + (k-1)(k-2)(n-1) + k(k-2)(n-1)(n-2)$ = $k(k-2)(n-1)^{2} + n - 1$.

Now we show that $hc(G)$ $\geq$ $k(k-2)(n-1)^{2} + n - 1$. Let $c$ be a minimum hamiltonian coloring of $G$. Since $G$ contains no hamiltonian path no two vertices can be colored the same. A hamiltonian coloring induces an ordering on $V(G)$ with increasing images. We may assume that 0 = $c(u_{0})$ $<$ $c(u_{1})$ $<$ ... $<$ $c(u_{p-1})$. We consider three subcases.

\textbf{Subcase - 1:} $c(w)$ = $c(u_{0})$ = 0.~~Since $D(u_{0}, u_{1})$ = $D(w, u_{1})$ = $n-1$, $c(u_{1}) \geq (k-1)(n-1)$. Also $D(v_{x}^{j}, v_{y}^{j})$ = $n-1$ for $1 \leq j \leq k$, $x \neq y$ and $D(v_{x}^{t}, v_{y}^{l})$ for $1 \leq x, y \leq n-1$, $1 \leq t, l (t \neq l) \leq k$. It follows that $c(u_{3})$ $\geq$ $(k-2)(n-1)$ and $c(u_{i+1})$ $\geq$ $c(u_{i}) + (k-2)(n-1)$ for all $3 \leq i \leq k(n-1)-1$. This implies that $c(u_{p-1})$ $\geq$ $(k-1)(n-1) + (k(n-1)-1)(k-2)(n-1)$ = $k(k-2)(n-1)^{2} + n - 1$. Therefore $hc(c)$ $\geq$ $k(k-2)(n-1)^{2} + n - 1$.

\textbf{Subcase - 2:} $c(w)$ = $c(u_{p-1})$ = $hc(c)$.~~Since $D(v_{x}^{j}, v_{y}^{j})$ = $n-1$ for $1 \leq j \leq k$, $x \neq y$ and $D(v_{x}^{t}, v_{y}^{l})$ for $1 \leq x, y \leq n-1$, $1 \leq t, l (t \neq l) \leq k$. For each $i$ with $1 \leq i \leq k(n-1)-1$, $c(v_{1})$ = 0 and $c(v_{i+1})$ = $(k-2)(n-1)$ and $c(w)$ = $c(u_{p-1})$ = $c(u_{k(n-1)-1}) + (k-1)(n-1)$. This implies that $c(u_{p-1})$ = $c(w)$ $\geq$ $(k(n-1)-1)(k-2)(n-1) + (k-1)(n-1)$ = $k(k-2)(n-1)^{2} + n - 1$. Therefore $hc(c)$ $\geq$ $k(k-2)(n-1)^{2} + n - 1$.

\textbf{Subcase - 3:} $c(u_{i})$ $\leq$ $c(w)$ $\leq$ $c(u_{i+1})$ for some $i$ with $1 \leq i \leq k(n-1)-1$.~~Since $D(v_{x}^{j}, v_{y}^{j})$ = $n-1$ for $1 \leq j \leq k$, $x \neq y$ and $D(v_{x}^{t}, v_{y}^{l})$ = $n-2$ for $1 \leq x, y \leq n-1$, $1 \leq t, l (t \neq l) \leq k$. Define $c(u_{0})$ = 0 and $c(u_{i+1})$ = $c(u_{i}) + (k-2)(n-1)$ for $1 \leq i \leq m$ and $m \leq p-3$. Then $c(u_{m+1})$ = $c(w)$ = $c(u_{m}) + (k-1)(n-1)$ and $c(u_{m+2})$ = $c(u_{m+1}) + (k-1)(n-1)$, $c(u_{i+1})$ = $c(u_{i}) + (k-2)(n-1)$ for $m+2 \leq i \leq p-1$. Therefore $hc(c)$ $\geq$ $k(k-2)(n-1)^{2} + 2(n-1)$.

Hence from Subcase - 1, 2 and 3, $hc(G)$ = $k(k-2)(n-1)^{2} + n - 1$.

Thus, from Case - 1 and 2, we have
\begin{eqnarray*}
hc(K_{n}^{k}) & = &\left\{
\begin{array}{ll}
(n-1)^{2}, & $if $ k = 2, \\ [0.3cm]
k(k-2)(n-1)^{2}+n-1, & $if $ k \geq 3.
\end{array}
\right.
\end{eqnarray*}
\end{proof}

We now determine the hamiltonian chromatics number for $B_{n,k}(d)$ for $k \geq 2$, $n,d \geq 3$ using Theorem \ref{hc:main}. Note that $B_{3,2}(3)$, $B_{3,3}(3)$ and $B_{3,2}(4)$ block graphs does not satisfies condition (c) of Theorem \ref{hc:main} but it is easy to verify that the hamiltonian chromatic numbers for these three graphs are coincide with the numbers produce by the formula stated in Theorem \ref{sym:thm}. Moreover, labels assigned by Algorithms given in proof of Theorem \ref{sym:thm} is the optimal hamiltonian coloring for these graphs.

\begin{theorem}\label{sym:thm} Let $k \geq 1$, $n \geq 2$, $d \geq 3$ be integers, $r$ = $\lfloor\frac{d}{2}\rfloor$ and $\Phi_{r}(x)$ = $1+x+x^{2}+...+x^{r-1}$. Then $hc\left(B_{n+1,k+1}(d)\right)$
\begin{eqnarray}\label{sym:eqn}
& = & \left\{
\small{
\begin{array}{ll}
n^{2}(k+1)\left[\Phi_{r}(kn)\left((k+1)\Phi_{r}(kn)-2r\right)
+\frac{2\left(\Phi_{r}(kn)-r\right)}{kn-1}\right]+n, & $if $ d $ is even$, \vspace{0.3 cm} \\
kn^{2}(n+1)\left[\Phi_{r}(kn)\left(k(n+1)\Phi_{r}(kn)-2r+1\right)
+\frac{2\left(\Phi_{r}(kn)-r\right)}{kn-1}\right], & $if $ d $ is odd$.
\end{array}}
\right.
\end{eqnarray}
\end{theorem}
\begin{proof} The order $p$ and total detour level of $B_{n+1,k+1}(d)$ is given by

\begin{eqnarray}\label{sym:p}
p = &\left\{
\begin{array}{ll}
1+\displaystyle\sum_{i = 1}^{r} (k+1)k^{i-1}n^{i}, & $if $ d $ is even$, \\
1+n+\displaystyle\sum_{i = 1}^{r}k^{i}n^{i+1}, & $if $ d $ is odd$.
\end{array}
\right.
\end{eqnarray}

\begin{eqnarray}\label{sym:l}
\mathcal{L}(G) & = &\left\{
\begin{array}{ll}
n^{2}(k+1)\left(r\Phi_{r}(kn)+\frac{r-\Phi_{r}(kn)}{kn-1}\right), & $if $ d $ is even$, \vspace{0.3 cm} \\
kn^{2}(n+1)\left(r\Phi_{r}(kn)+\frac{r-\Phi_{r}(kn)}{kn-1}\right), & $if $ d $ is odd$.
\end{array}
\right.
\end{eqnarray}
Substituting (\ref{sym:p}) and (\ref{sym:l}) into (\ref{eqn:lower}) gives the right hand side of (\ref{sym:eqn}).

We now prove that the right hand side of (\ref{sym:eqn}) is the actual value for the hamiltonian chromatics number of symmetric block graph. For this purpose we give a systematic hamiltonian coloring whose span is the right hand side of (\ref{sym:eqn}). We consider following two cases.

\textbf{Case - 1:} $d$ is even.~~In this case, symmetric block graphs has only one detour central vertex say $w$. We apply the following algorithm to find a hamiltonian coloring of symmetric block graph of even diameter whose span is right-hand side of (\ref{sym:eqn}).

\noindent\textbf{Algorithm 1:} An optimal hamiltonian coloring of symmetric block graphs $B_{n+1,k+1}(d)$, where $d$ is even. \\
\textbf{Input:} A symmetric block graph $B_{n+1,k+1}$ of even diameter. \\
\textbf{Idea:} Find an ordering of vertices of block graphs $B_{n+1,k+1}$ of even diameter which satisfies Theorem \ref{hc:main} and labeling defined by (\ref{f0})-(\ref{f1}) is a hamiltonian coloring whose span is right-hand side of (\ref{sym:eqn}). \\
\textbf{Initialization:} Start with a central vertex $w$. \\
\textbf{Iteration:} Define $c$ : $V(B_{n+1,k+1})$ $\rightarrow$ \{0,1,2,..\} as follows: \\
\textbf{Step-1:} Let $v^{1}$, $v^{2}$, ..., $v^{(k+1)n}$ be the vertices adjacent to $w$ such that any $k+1$ consecutive vertices in the list are in different blocks. \\
\textbf{Step-2:} Now $kn$ descendent vertices of each $v^{t}$, $t = 1,2,...,(k+1)n$ by $v_{0}^{t}$, $v_{1}^{t}$,...,$v_{kn-1}^{t}$ such that any $k$ consecutive vertices in the list are in different blocks. Next the $kn$ descendent vertices of each $v_{l}^{t}$, $0 \leq l \leq kn-1, 1 \leq t \leq (k+1)n$ by $v_{l0}^{t}$, $v_{l1}^{t}$,...,$v_{l(kn-1)}^{t}$ such that any $k$ consecutive vertices lies in different blocks; inductively $kn$ descendent vertices of $v_{i_{1},i_{2},...,i_{l}}^{t}$ ($0 \leq i_{1},i_{2},...,i_{l} \leq kn-1$, $1 \leq t \leq (k+1)n$) are indexed by $v_{i_{1},i_{2},...,i_{l},i_{l+1}}^{t}$ where $i_{l+1}$ = 0, 1, ..., $kn-1$ such that any $k$ consecutive vertices in the list are in different blocks. \\
\textbf{Step-3:} Rename $v_{i_{1},i_{2},...,i_{l},i_{l+1}}^{t}$ by $v_{j}^{t}$, ($1 \leq t \leq (k+1)n$), where \\

$j$ = $1 + i_{1} + i_{2}(kn) + ... + i_{l}(kn)^{l-1} + \displaystyle\sum_{l+1 \leq t \leq \frac{d}{2}} (kn)^{t}$. \\
\textbf{Step-4:} Give ordering \{$u_{0}$, $u_{1}$,...,$u_{p-1}$\} of vertices of symmetric block graphs as follows. \\
For $1 \leq j \leq p-(k+1)n-1$, let
\begin{eqnarray}
u_{j} := &\left\{
\begin{array}{ll}
v_{s}^{t}, \mbox{ where } s = \lceil \frac{j}{(k+1)n} \rceil, \mbox{ if } j \equiv t (\mbox{mod }(k+1)n), 1 \leq t \leq (k+1)n-1, \\ [0.3cm]
v_{s}^{(k+1)n}, \mbox{ where } s = \lceil \frac{j}{(k+1)n} \rceil, \mbox{ if } j \equiv 0 (\mbox{mod }(k+1)n)
\end{array} \non
\right.
\end{eqnarray}
For $p-(k+1)n \leq j \leq p-1$, let
$$
\begin{array}{l}
u_{j} := v^{j-p+(k+1)n+1}. \\
\end{array}
$$
Then above defined ordering \{$u_{0}$, $u_{1}$,...,$u_{p-1}$\} of vertices satisfies Theorem \ref{hc:main}. \\
\textbf{Step-5:} Define $c$ : $V(B_{n+1,k+1})$ $\rightarrow$ \{0,1,2,...\} by $c(u_{0})$ = 0 and $c(u_{i+1})$ = $c(u_{i}) + p - 1 - \mathcal{L}(u_{i}) - \mathcal{L}(u_{i+1}) - \omega + 1$, $0 \leq i \leq p-2$. \\
\textbf{Output:} The span of $c$ is span($c$) = $c(u_{p-1})$ = $c(u_{0}) + (p-1)^{2} - 2\left(\displaystyle\sum_{u \in V(G)} \mathcal{L}(u)\right) + n - 1$ = $(p-1)^{2} - 2 \mathcal{L}(B_{n+1,k+1}) + n - 1$ which is exactly the right-hand side of (\ref{sym:eqn}) by using (\ref{sym:p}) and (\ref{sym:l}) in the case of symmetric block graphs.

\textbf{Case - 2:} $d$ is odd.~~In this case, symmetric block graphs has $n+1$ central vertices say $v^{1}$, $v^{2}$, ..., $v^{n+1}$. We apply the following algorithm to find a hamiltonian coloring of symmetric block graph of odd diameter whose span is right-hand side (\ref{sym:eqn}).

\noindent\textbf{Algorithm 2:} An optimal hamiltonian coloring of symmetric block graphs $B_{n+1,k+1}(d)$, where $d$ is odd. \\
\textbf{Input:} A symmetric block graph $B_{n+1,k+1}(d)$ of odd diameter $d$. \\
\textbf{Idea:} Find an ordering of vertices of block graphs $B_{n+1,k+1}$ of odd diameter which satisfies Theorem \ref{hc:main} and labeling defined by (\ref{f0})-(\ref{f1}) is a hamiltonian coloring whose span is right-hand side of (\ref{sym:eqn}). \\
\textbf{Initialization:} Starts with central vertices $v^{1}$, $v^{2}$, ..., $v^{n+1}$. \\
\textbf{Iteration:} Define $c$ : $V(B_{n+1,k+1})$ $\rightarrow$ \{0,1,2,...\} \\
\textbf{Step-1:} Now $kn$ descendent vertices of each $v^{t}$, $t = 1,2,...,(k+1)n$ by $v_{0}^{t}$, $v_{1}^{t}$,...,$v_{kn-1}^{t}$ such that any $k$ consecutive vertices in the list are in different blocks. Next $kn$ descendent vertices of each $v_{l}^{t}$, $0 \leq l \leq kn-1, 1 \leq t \leq n+1$ by $v_{l0}^{t}$, $v_{l1}^{t}$,...,$v_{l(kn-1)}^{t}$ such that any $k$ consecutive vertices are in different blocks; inductively $kn$ descendent vertices of $v_{i_{1},i_{2},...,i_{l}}^{t}$ ($0 \leq i_{1},i_{2},...,i_{l} \leq kn-1$, $1 \leq t \leq n+1$) are indexed by $v_{i_{1},i_{2},...,i_{l},i_{l+1}}^{t}$ where $i_{l+1}$ = 0, 1, ..., $kn-1$ such that any $k$ consecutive vertices in the list are in different blocks. \\
\textbf{Step-2:} We rename $v_{i_{1},i_{2},...,i_{l},i_{l+1}}^{t}$ by $v_{j}^{t}$, ($1 \leq t \leq n+1$), where \\

$j$ = $1 + i_{1} + i_{2}(kn) + ... + i_{l}(kn)^{l-1} + \displaystyle\sum_{l+1 \leq t \leq \frac{d-1}{2}} (kn)^{t}$. \\
\textbf{Step-3:} Define an ordering \{$u_{0}$, $u_{1}$,...,$u_{p-1}$\} as follows: \\
For $1 \leq j \leq p-n-1$, let
\begin{eqnarray}
u_{j} := &\left\{
\begin{array}{ll}
v_{s}^{t}, \mbox{ where } s = \lceil j/(n+1) \rceil, \mbox{ if } j \equiv t (\mbox{mod }(n+1)) \mbox{ with } 1 \leq t \leq n, \\ [0.3cm]
v_{s}^{n+1}, \mbox{ where } s = \lceil j/(n+1) \rceil, \mbox{ if } j \equiv 0 (\mbox{mod }(n+1))
\end{array} \non
\right.
\end{eqnarray}
For $p-n \leq j \leq p-1$, \\
$$
\begin{array}{l}
u_{j} := v^{j-p+n+1}. \\
\end{array}
$$
Then above defined ordering \{$u_{0}$, $u_{1}$,...,$u_{p-1}$\} of vertices satisfies Theorem \ref{hc:main}.
\textbf{Step-4:} Define $c$ : $V(B_{n+1,k+1})$ $\rightarrow$ \{0,1,2,...\} by $c(u_{0})$ = 0 and $c(u_{i+1})$ = $c(u_{i}) + p - 1 - L(u_{i}) - L(u_{i+1}) - \omega + 1$, $0 \leq i \leq p-1$. \\
\textbf{Output:} The span of $c$ is $c(u_{p-1})$ = $hc(G)$ = $c(u_{0}) + (p-1)^{2} - 2\left(\displaystyle\sum_{u \in V(G)} \mathcal{L}(u)\right) - (p-1)(n-1)$ = $(p - 1)^{2}$ $-$ 2 $\mathcal{L}(G)$ $-$ $(p - 1)(n - 1)$ which is exactly the right-hand side of (\ref{sym:eqn}) by using (\ref{sym:p}) and (\ref{sym:l}) in the case of symmetric block graphs.

\noindent Thus, from Case - 1 and Case - 2, we obtain $hc\left(B_{n+1,k+1}(d)\right)$
\begin{eqnarray*}
& = & \left\{
\small{
\begin{array}{ll}
n^{2}(k+1)\left[\Phi_{r}(kn)\left((k+1)\Phi_{r}(kn)-2r\right)
+\frac{2\left(\Phi_{r}(kn)-r\right)}{kn-1}\right]+n, & $if $ d $ is even$, \vspace{0.3 cm} \\
kn^{2}(n+1)\left[\Phi_{r}(kn)\left(k(n+1)\Phi_{r}(kn)-2r+1\right)
+\frac{2\left(\Phi_{r}(kn)-r\right)}{kn-1}\right], & $if $ d $ is odd$.
\end{array}}
\right.
\end{eqnarray*}
\end{proof}

\begin{example} An optimal hamiltonian coloring of $B_{4,2}(4)$ using the procedure of Theorem \ref{sym:thm} is shown in Figure \ref{Fig1}-(a).
\end{example}
For $B_{4,2}(4)$, $k$ = 1, $n$ = 3, $d$ = 4, $r$ = $\lfloor\frac{d}{2}\rfloor$ = 2 and $\Phi_{\lfloor\frac{d}{2}\rfloor}(kn)$ = $\Phi_{2}(3)$ = 1 + 3 = 4. By Theorem \ref{sym:thm}, $hc(B_{4,2}(4))$ \\ = $n^{2}(k+1)\left[\Phi_{\lfloor\frac{d}{2}\rfloor}(kn) \left((k+1) \Phi_{\lfloor\frac{d}{2}\rfloor}(kn) - 2\lfloor\frac{d}{2}\rfloor\right)
+\frac{2\left(\Phi_{\lfloor\frac{d}{2}\rfloor}(kn)-\lfloor\frac{d}{2}\rfloor\right)}{kn-1}\right]+n$ \\ =  $3^{2}\cdot(1+1)\left[4((1+1)\cdot4-2\cdot2)+\frac{2(4-2)}{3-1}\right] + 3$ = 327.

\begin{example} An optimal hamiltonian coloring of $B_{4,2}(5)$ using the procedure of Theorem \ref{sym:thm} is shown Figure \ref{Fig1}-(b).
\end{example}
For $B_{4,2}(5)$, $k$ = 1, $n$ = 3, $d$ = 5, $r$ = $\lfloor\frac{d}{2}\rfloor$ = 2 and $\Phi_{\lfloor\frac{d}{2}\rfloor}(kn)$ = $\Phi_{2}(3)$ = 1 + 3 = 4. By Theorem \ref{sym:thm}, $hc(B_{4,2}(4))$ \\ =
$kn^{2}(n+1)\left[\Phi_{\lfloor\frac{d}{2}\rfloor}(kn)\left(k(n+1) \Phi_{\lfloor\frac{d}{2}\rfloor}(kn) - 2\lfloor\frac{d}{2}\rfloor + 1\right)  + \frac{2\left(\Phi_{\lfloor\frac{d}{2}\rfloor}(kn)-\lfloor\frac{d}{2}\rfloor\right)}{kn-1}\right]$ \\ =  $1\cdot3^{2}\cdot(3+1)\left[4(1\cdot(3+1)\cdot4-2\cdot2+1)+\frac{2(4-2)}{3-1}\right]$ = 1944.

\begin{figure}[h!]
\begin{center}
  \includegraphics[width=5.5 in]{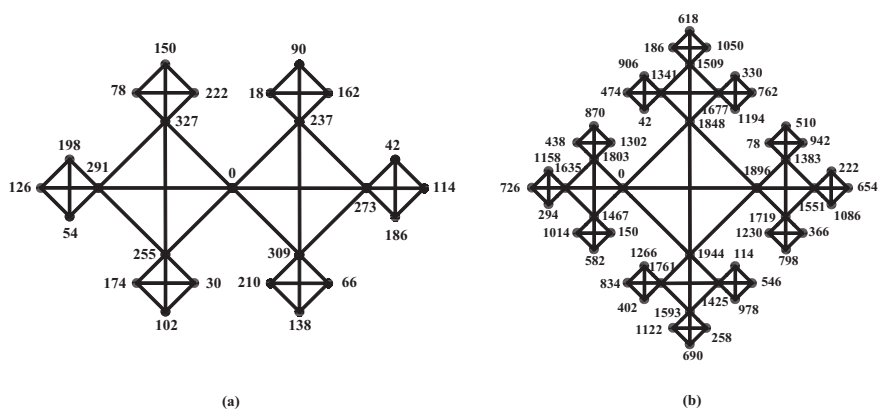}
  \caption{Optimal hamiltonian coloring of $B_{4,2}(4)$ and $B_{4,2}(5)$.}\label{Fig1}
\end{center}
\end{figure}

\end{document}